\documentclass[10pt]{amsart}
\usepackage{amsmath}

\theoremstyle{plain}
\newtheorem{theorem}{Theorem}[section]
\newtheorem{proposition}[theorem]{Proposition}
\newtheorem{lemma}[theorem]{Lemma}
\newtheorem{corollary}[theorem]{Corollary}

\theoremstyle{remark}
\newtheorem{definition}[theorem]{Definition}
\newtheorem{remark}[theorem]{Remark}
\newtheorem{example}[theorem]{Example}
\newtheorem{question}[theorem]{Question}

\newcommand{\B}{\mathcal{B}}
\newcommand{\BH}{\mathcal{B(H)}}
\newcommand{\bC}{\mathbb{C}}

\newcommand{\FH}{\mathcal{F(H)}}

\newcommand{\KH}{\mathcal{K(H)}}

\newcommand{\BPH}{\mathcal{B}_p\mathcal{(H)}}

\newcommand{\cH}{{\mathcal{H}}}
\newcommand{\SH}{\mathcal{S(H)}}
\newcommand{\GH}{\mathcal{G(H)}}
\newcommand{\FC}{\mathcal{F}_C}
\newcommand{\SC}{\mathcal{S}_C}
\newcommand{\SCP}{\mathcal{S}_{C,p}}
\newcommand{\SCY}{\mathcal{S}_{C,1}}
\newcommand{\SCO}{\mathcal{S}_{C,0}}
\newcommand{\SCq}{\mathcal{S}_{C,q}}
\newcommand{\SSC}{\mathcal{O}_C}

\newcommand{\cA}{\mathcal{A}}
\newcommand{\cJ}{\mathcal{J}}

\newcommand{\cI}{\mathcal{I}}

\newcommand{\eps}{\varepsilon}

\newcommand{\ind}{\textup{ind}}

\newcommand{\cS}{\mathcal{S}}
\newcommand{\la}{\langle}

\newcommand{\onb}{\textsc{onb}}
\newcommand{\ra}{\rangle}

\newcommand{\ran}{\textup{ran}}

\newcommand{\rank}{\textup{rank}}
\newcommand{\tr}{\textup{tr}}

\begin{document}

\title[The Jordan algebra of CSOs]
{The Jordan algebra of complex symmetric operators}

\author[C. Wang]{Cun Wang}
\address{School of Mathematics and Statistics\\Beijing Institute of Technology\\Beijing 102488\\P. R. China}
\email{wangcun18@mails.jlu.edu.cn}
\thanks{This work is part of the first author's doctoral dissertation at Jilin University.}

\author[S. Zhu]{Sen Zhu}
\address{Department of Mathematics\\Jilin University\\Changchun 130012\\P. R. China}
\email{zhusen@jlu.edu.cn}

%






\begin{abstract}
For a conjugation $C$ on a separable, complex Hilbert space $\mathcal{H}$, the set $\mathcal{S}_C$ of $C$-symmetric operators on $\mathcal{H}$ forms a weakly closed, selfadjoint, Jordan operator algebra. In this paper we study $\mathcal{S}_C$ in comparison with the algebra $\mathcal{B(H)}$ of all bounded linear operators on $\mathcal{H}$, and obtain $\mathcal{S}_C$-analogues of some classical results concerning $\mathcal{B(H)}$.
We determine the Jordan ideals of $\mathcal{S}_C$ and their dual spaces. Jordan automorphisms of $\mathcal{S}_C$ are classified. We determine the spectra of Jordan multiplication operators on $\mathcal{S}_C$ and their different parts. It is proved that those Jordan invertible ones constitute a dense, path connected subset of $\mathcal{S}_C$.
\end{abstract}

\maketitle



\section{Introduction}

The aim of this paper is to study a class of Jordan operator algebras consisting of complex symmetric operators. Before proceeding let us first recall a few definitions.

Throughout this paper, $\cH$ $(\cH_0,\cH_1,\cH_2,\cdots$, etc.) will always
denote a separate, complex Hilbert space with an inner product $\la\cdot,\cdot\ra$.
The Banach space of all bounded linear operators mapping  $\cH_1$ into $\cH_2$ will be denoted by $\B(\cH_1,\cH_2)$. For convenience, we shall write $\BH$ for $\B(\cH,\cH)$.

\begin{definition}
A map $C:\cH \rightarrow\cH$ is called a {\it conjugation} if
\begin{enumerate}
  \item[(i)] $C$ is antilinear, i.e. $C(\alpha x+y)=\overline{\alpha} C x+ Cy$ for $x,y\in\cH$ and $\alpha\in\bC$,
  \item[(ii)] $C$ is invertible with $C^{-1}=C$, and
  \item[(iii)] $\la Cx, Cy\ra=\la y,x\ra$ for all $x, y\in\cH$.
\end{enumerate}
\end{definition}

\begin{definition}
An operator $T\in\BH$ is said to be {\it complex symmetric} (c.s.) if
$CTC=T^*$ for some conjugation $C$ on $\cH$; in this case, $T$ is called {\it $C$-symmetric}.
\end{definition}

The term ``symmetric" stems from the fact that an operator is complex symmetric if and only if there is an orthonormal basis
$\{e_i\}$ of $\cH$ such that $\la Te_i,e_j\ra=\la Te_j,e_i\ra$ for all $i,j$, that is, $T$ can be written as a (complex) symmetric matrix relative to $\{e_i\}$.
Following Garcia and Wogen \cite{Gar09}, we let $\SH$ denote the collection of c.s. operators on $\cH$.
Many important examples of c.s. operators have been identified, such as normal operators, binormal operators, Hankel operators, truncated Toeplitz operators, certain partial isometries, weighted shifts, composition operators and many integral operators (see \cite{Gar06,Gar07,Gar09,Noor,Maofa}).

The study of abstract c.s. operators, initiated in \cite{Gar06,Gar07}, is relatively new,
although it has classical roots in the work on automorphic functions \cite{Hua},
projective geometry \cite{Jacobson}, quadratic forms \cite{Schur}, symplectic geometry \cite{Siegel}, function theory \cite{Takagi} and extension theory of differential operators \cite{Glazman}. People's current interest in c.s. operators are chiefly inspired by many interesting results of Garcia, Putinar and Wogen \cite{Gar06,Gar07,Gar09} concerning the structure of c.s. operators as well as their connections to concrete operators \cite{Clark,Gar2007,Foguel}
 and applications to mathematics physics \cite{Gar14,Putinar,Prodan}. In particular, the study of c.s. operators is closely related to that of truncated Toeplitz operators, which was initiated in Sarason's seminal paper \cite{Sarason}. In their recent paper \cite{Bercov}, Bercovici and Timotin showed that truncated Toeplitz operators can be characterized by a collection of complex symmetries.

Recently there has been a growing interest in the algebraic aspect of c.s. operators.
In \cite{GuoJiZhu}, certain connections between c.s.
operators and anti-automorphisms of singly generated $C^*$-algebras are established. Furthermore, a general answer to the norm closure problem for c.s. operators was provided in \cite{ZhuMA} which relies on an intensive analysis of singly generated $C^*$-algebras. In \cite{ShenZhu,ZhuJMAA}, c.s. generator problem for operator algebras has been studied. In a recent paper \cite{BW-Invol}, Blecher and Wang studied involutive operator algebras and obtained a characterization of operator algebras with linear involutions in terms of c.s. operators.

Since $\SH$ is not closed under addition or multiplication, we may pay attention to a typical linear subspace of $\BH$ contained in $\SH$. Given a conjugation $C$ on $\cH$, we denote by $\SC$ the set of all $C$-symmetric operators on $\cH$. Clearly, $\SC$ is a subspace of $\BH$ closed in the weak operator topology. Note that $\SH=\cup_C \SC$, where the union is taken over all conjugations on $\cH$.
Later we shall show that if $C_1,C_2$ are conjugations on $\cH$, then there exists a unitary operator $U$ on $\cH$
such that $U\cS_{C_1} U^*=\cS_{C_2}$ (see Corollary \ref{C:unitary}). So, up to unitary equivalence, $\SC$ contains all c.s. operators on $\cH$.
Throughout the following, we always assume that $C$ is a conjugation on $\cH$, unless stated otherwise. Also, we always assume that $\dim\cH=\infty$, since
almost all results in finite dimensional case follow readily from that in infinite dimensional case.

Some interesting results concerning $\SC$ as a subspace of $\BH$ have been obtained.
For example, Garcia {\cite[Theorem 1]{Gar2007}} proved that each contraction in $\SC$ is the mean of two unitary ones in $\SC$,
whose proof relies on a refined polar decomposition for c.s. operators obtained by Garcia and Putinar \cite{Gar07}.
This result is a $C$-symmetric analogue of a classical result in $\BH$.
In \cite{Danciger}, Danciger, Garcia and Putinar established for $\SC$ a natural analogue of Courant's minimax principle to estimate singular values of compact ones in $\SC$. It was shown in \cite{PtakRef} that $\SC$ is transitive and $2$-hyperreflexive.

The aim of this paper is to study $\SC$ in the Jordan setting.
This is inspired by the observation that $\SC$ is closed under the Jordan product $\circ$, defined by
$$A\circ B=\frac{1}{2}(AB+BA),\ \ \ \forall A,B\in\BH.$$
That is, $\SC$ forms a Jordan operator algebra. By a Jordan operator algebra we mean a norm-closed subspace of $\BH$ closed under the Jordan product $\circ$.
Thus $\SC$ is a Jordan subalgebra of $\BH$.

Jordan algebras arose from the search for a new algebraic setting for quantum
mechanics \cite{JordanVonNeuWig}, and turned out to have illuminating connections
with many areas of mathematics. Any associative algebra gives rise to a Jordan algebra under the Jordan product $\circ$. Note that $A\in\SC$ implies $A^*\in \SC$. Hence $\SC$ is also selfadjoint. Selfadjoint Jordan operator algebras are known as JC$^*$-algebras \cite{Wright}.

We remark that $\SC$ lies naturally in many more general contexts.
In fact, $\SC$ has been studied under the name of Hermitian type Cartan factor for many years.
There are six types of Cartan factors, namely rectangular type, Hermitian type, symplectic type, triple spin
factors and two finite-dimensional exceptional Cartan factors.
Cantor factors arose in \'{E}. Cartan's classification of finite dimensional bounded symmetric domains (see \cite{Cartan35} or \cite[Theorem 2.5.9]{ChuCH}) and play an important role in the proof of the Gelfand-Naimark theorem for JB$^*$-triples \cite{FriedmanRusso}. Any JB$^*$-triple is triple isomorphic to a closed subtriple of an $l_\infty$-sum of Cartan factors (\cite[Theorem 3.3.19]{ChuCH}).


This paper focuses on the special Jordan operator algebra $\SC$, and aims to present some applications of recent advances in the theory of c.s. operators to the Jordan structure of $\SC$. We shall prove some results concerning the algebraic aspect of the space $\SC$, including Jordan ideals, Schatten $p$-classes, automorphisms, invertibles, multiplication operators. We obtain $\SC$-analogues of some classical results concerning $\mathcal{B(H)}$, showing that $\SC$ is far more similar to the associative operator algebra $\BH$ than was suspected.

In Section \ref{S:ideal}, we study the Jordan ideal structure of $\SC$.
A linear subspace $\mathcal{J}$ of $\SC$ is called a {\it Jordan ideal} if $A\circ X\in \mathcal{J}$ for every $A\in \SC$ and $X\in \mathcal{J}$.
Note that Jordan ideals of $\SC$ coincide with its triple ideals as a Jordan triple system \cite[page 38]{ChuCH}.
We shall show that Jordan ideals of $\SC$ arise by intersection from associative ideals of $\BH$ and hence are selfadjoint (see Theorem \ref{T:Jordan}).
By an associative ideal of $\BH$, we mean a two-sided ideal (not necessarily norm closed) of $\BH$ under the usual multiplication of operators. Note that Jordan ideals of $\BH$ coincide with its associative ideals (see \cite{Fong1982}). So our result shows that $\SC$ and $\BH$ have the same Jordan ideal structure. Moreover, we observe that $\SC$ is universal in the sense that any Jordan subalgebra of $\BH$ is Jordan isomorphic to a Jordan subalgebra of $\SC$ (see Corollary \ref{C:universal}). This inspires us to examine whether some classical facts about $\BH$ still hold or have analogues in $\SC$.


In Subsection 2.2, we study those norm ideals of $\SC$ induced by Schatten $p$-classes.
The Schatten $p$-class of compact operators on $\cH$ is denoted by $\BPH$, $1\leq p<\infty$. In addition we write $\B_0(\cH)$ for the collection
of compact operators on $\cH$. Denote $\cS_{C,p}=\SC\cap\BPH$.  We establish in Subsection 2.2 the dual relations among $\SC$ and its Jordan ideals $\cS_{C,p} (p\in\{0\}\cup [1,\infty))$ (see Proposition \ref{P:SchattenP}).
As an application, we classify Jordan automorphisms of $\SC$ and show that Jordan automorphisms of $\SC$ are implemented by those unitary operators on $\cH$ commuting with $C$ (see Theorem \ref{T:JordanIso}).

In Section \ref{S:JordanMult}, we concentrate on Jordan multiplication operators on $\SC$.
For $T\in\SC$, we define $J_T: \SC\rightarrow \SC$ as $J_T(X)=T\circ X$ for $X\in\SC$. We call $J_T$ the
{\it Jordan multiplication operator} with symbol $T$. Note that $J_T$ is called {\it left multiplication by} $T$ in \cite{ChuCH}, and is called the {\it Jordan translation} of $T$ in \cite{Topping}. These operators play a basic role in the study of Jordan algebras. In particular, quadratic product, trilinear product as well as the centre can be defined in terms of them. We shall determine the spectra of Jordan multiplication operators $J_T$ and their restrictions to some Jordan ideals of $\SC$ (see Theorem \ref{T:JMultiply} and Corollary \ref{C:ideal}).
One shall see that their proofs rely on the dual relations among $\SC$ and its Jordan ideals $\cS_{C,p} (p\in[1,\infty])$.

In Section \ref{S:invertible}, we discuss Jordan invertible elements in $\SC$.
For $T\in\SC$, we define the {\it quadratic operator} $Q_T:\SC \rightarrow \SC$ by
$Q_T(X)=TXT$ for $X\in\SC$. An element $T$ of $\SC$ is called {\it Jordan invertible} if
$Q_T$ is invertible as a bounded linear operator on $\SC$ with $Q_T^{-1}=Q_{A}$ for some $A\in\SC$ (\cite[page 107]{ChuCH}).
One can check that an element $T\in\SC$ is Jordan invertible if and only if $T$ is bijective or, equivalently, $T$ is an invertible bounded operator  on $\cH$.
It is proved that those invertible ones in $\mathcal{S}_C$ constitute a dense, path connected subset of $\mathcal{S}_C$.
This is an analogue of a result of C. Apostol, L. Fialkow, D. Herrero and D. Voiculescu concerning invertible approximation in $\BH$ (see \cite[Proposition 10.1]{AFHV}).

Results obtained in this paper suggest a rich structure theory of $\SC$, provide interesting contrasts between $\SC$ and $\BH$, and also show that the structure of $\SC$ deserves further study.
Concerning our methodology, most techniques used in this paper are developed based on some recent results concerning the structure of c.s. operators. Also some of our results employ tools from noncommutative approximation of Hilbert space operators, as represented in the two-volume monographs of Apostol-Fialkow-Herrero-Voiculescu \cite{AFHV,Herr89}.

\section{Jordan ideals of $\SC$}\label{S:ideal}

This section focuses on Jordan ideals of $\SC$. It is proved that each proper Jordan ideal of $\SC$ is the intersection of $\SC$ with some associative
ideal of $\BH$ and hence consists of some compact operators on $\cH$. We shall extend some results concerning compact operators in $\BH$ to $\SC$.

\subsection{Characterization of Jordan ideals}
The main result of this subsection is the following theorem, which
shows that each Jordan ideal of $\SC$ is induced by an associative ideal of $\BH$.

\begin{theorem}\label{T:Jordan}
A subset $\cJ$ of $\SC$ is a Jordan ideal of $\SC$ if and only if $\cJ=\cI\cap \SC$ for some associative ideal $\cI$ of $\BH.$
\end{theorem}

To prove Theorem \ref{T:Jordan}, we first make some preparation.

Let $C$ be a conjugation on $\cH$. For $X\in\BH$, we denote $X^t=CX^*C$.
Define $D: \cH^{(2)} \rightarrow \cH^{(2)}$ as
\[D:(x_1,x_2)\longmapsto (Cx_2, Cx_1), \ \ \ \forall (x_1,x_2)\in \cH^{(2)}.\]
Then $D$ can be written as
$$D=\begin{bmatrix}
0&C\\
C&0
\end{bmatrix} \begin{matrix}
  \cH\\ \cH
\end{matrix} .$$ One can check that $D$ is a conjugation on $\cH^{(2)}$.

\begin{lemma}\label{L:Jordan}
Let $C$ be a conjugation on $\cH$ and $$D=\begin{bmatrix}
0&C\\
C&0
\end{bmatrix}.$$Assume that $T\in\B(\cH^{(2)})$ and
\[ T=\begin{bmatrix}
A&E\\
F&B
\end{bmatrix} \begin{matrix}
  \cH\\ \cH
\end{matrix}.\] Then
\begin{enumerate}
\item[(i)]  $T\in \mathcal{S}_D$ if and only if $B=A^t$, $E=E^t$ and $F=F^t$;
\item[(ii)] if $\cJ$ is a Jordan ideal of $\cS_D$ and $T\in \cJ$, then $\cJ$ contains the following operators on $\cH^{(2)}$
$$\begin{bmatrix}
0&0\\
F&0
\end{bmatrix},\begin{bmatrix}
0&E\\
0&0
\end{bmatrix},\begin{bmatrix}
A&0\\
0&B
\end{bmatrix},\begin{bmatrix}
B &0\\
0&A
\end{bmatrix},  \begin{bmatrix}
0&A+B\\
0&0
\end{bmatrix},\begin{bmatrix}
0&0\\
A+B&0
\end{bmatrix},\begin{bmatrix}
E&0\\
0&E
\end{bmatrix},\begin{bmatrix}
F&0\\
0&F
\end{bmatrix}.$$
\end{enumerate}
\end{lemma}

\begin{proof}
(i) By the definition, $T\in \mathcal{S}_D$ if and only if $DT=T^*D$.
The result follows from a direct matrix calculation.

(ii) Define $Y_1,Y_2\in\B(\cH^{(2)})$ as
\[Y_1=\begin{bmatrix}
0&0\\
I&0
\end{bmatrix},\ \  Y_2=\begin{bmatrix}
0&I\\
0&0
\end{bmatrix},\]
where $I$ is the identity operator on $\cH$.
Then, by (i), $Y_1,Y_2\in\cS_D$. 

It is easy to check that
\[ [Y_1\circ(Y_1\circ T)]\circ Y_2 =\frac{1}{4}\begin{bmatrix}
E&0\\
0&E
\end{bmatrix}, \ \ \left\{[Y_1\circ(Y_1\circ T)]\circ Y_2\right\}\circ Y_2=\frac{1}{4}\begin{bmatrix}
0&E\\
0&0
\end{bmatrix}\]
and
\[Y_1\circ [(T\circ Y_2)\circ Y_2] =\frac{1}{4}\begin{bmatrix}
F&0\\
0&F
\end{bmatrix}, \ \ Y_1\circ\left\{ Y_1\circ [(T\circ Y_2)\circ Y_2]\right\}=\frac{1}{4}\begin{bmatrix}
0&0\\
F&0
\end{bmatrix}.\] Thus
\[ \begin{bmatrix}
0&E\\
0&0
\end{bmatrix}, \begin{bmatrix}
0&0\\
F&0
\end{bmatrix}, \begin{bmatrix}
E&0\\
0&E
\end{bmatrix},\begin{bmatrix}
F&0\\
0&F
\end{bmatrix} \in \cJ. \]
It follows immediately  that $A\oplus B\in\cJ$,
\[Y_1\circ \begin{bmatrix}
A&0\\
0&B
\end{bmatrix}=\frac{1}{2}\begin{bmatrix}
0&0\\
A+B&0
\end{bmatrix}\in\cJ, \ \ \ \begin{bmatrix}
A&0\\
0&B
\end{bmatrix}\circ Y_2 =\frac{1}{2}\begin{bmatrix}
0& A+B\\
0&0
\end{bmatrix}\in\cJ,\]
\[ Y_1\circ\left(\begin{bmatrix}
A&0\\
0&B
\end{bmatrix}\circ Y_2\right) =\frac{1}{4}\begin{bmatrix}
  A+B&0\\
0&A+B
\end{bmatrix}\in\cJ.\] Therefore we conclude that $B\oplus A\in\cJ$.
\end{proof}


\begin{lemma}\label{L:unitary}
If $C_1,C_2$ are two conjugations on $\cH$, then there exists unitary $U\in\BH$ such that $U^*C_1U=C_2$.
\end{lemma}

\begin{proof}
Since $C_1,C_2$ are conjugations on $\cH$, by \cite[Lemma 2.11]{Gar14}, there exists two orthonormal bases $\{e_n: n\geq 1\}$ and $\{f_n: n\geq 1\}$ such that $C_1e_n=e_n$ and $C_2f_n=f_n$ for all $n$. Define a unitary operator $U$ on $\cH$ as $Uf_n=e_n$, $n\geq 1$.
Then it is easy to check that $U^*C_1U=C_2$.
\end{proof}

\begin{corollary}\label{C:unitary}
If $C_1,C_2$ are two conjugations on $\cH$, then there exists unitary $U\in\BH$ such that $\mathcal{S}_{C_2}=U^* \mathcal{S}_{C_1}U$.
\end{corollary}

The preceding result shows that $\SH$ is the union of all Jordan operator algebras unitarily equivalent to $\SC$.

\begin{corollary}\label{C:unitary-cntd}
The set of all conjugations on $\cH$ is path connected (as a subset of the Banach space of bounded antilinear operators on $\cH$).
\end{corollary}

\begin{proof}
Fix two conjugations $C_1$ and $C_2$ on $\cH$. By Lemma \ref{L:unitary}, we can find unitary $U\in\BH$ such that  $U^*C_1U=C_2$.
It is well known that the set $\mathcal{U(H)}$ of all unitary operators on $\cH$ is path connected. Then there exists continuous $V:[0,1]\rightarrow\mathcal{U(H)}$ such that
$V(0)=I$ and $V(1)=U$. Thus $\{V(\lambda)^*C_1V(\lambda); \lambda\in[0,1]\}$ is a path of conjugations from $C_1$  to $C_2$.
\end{proof}

\begin{proof}[Proof of Theorem \ref{T:Jordan}]
The sufficiency is obvious.

``$\Longrightarrow$". Set
$$D=\begin{bmatrix}
0&C\\
C&0
\end{bmatrix}.$$ Since $\dim\cH=\dim\cH^{(2)}$, in view of Corollary \ref{C:unitary}, there exists unitary $U:\cH\rightarrow\cH^{(2)}$ such that $U\SC U^*=\mathcal{S}_D$. Thus it suffices to prove the conclusion for $\mathcal{S}_D$.

Assume that $\cJ$ is a Jordan ideal of $\cS_D$.
Set \[\cJ_0= \left\{A\in\BH: \exists A_1,A_2,A_3\in \BH ~\textup{with} \begin{bmatrix}
A&A_1\\
A_2&A_3
\end{bmatrix}\in \cJ \right\}.\]

{\it Claim 1}. $\cJ_0$ is a Jordan ideal of $\BH$.

It is obvious that $\cJ_0$ is a linear subspace of $\BH$.
Now assume that $A\in \cJ_0$ and $B\in\BH$. So, by Lemma \ref{L:Jordan} (i), there exists an element $T\in \cJ$ with the form
\[T=\begin{bmatrix}
A&E\\
F&A^t
\end{bmatrix},\] where $E=E^t$ and $F=F^t$. By Lemma \ref{L:Jordan} (ii), we have
\[\begin{bmatrix}
A&0\\
0&A^t
\end{bmatrix}\in \cJ. \eqno(1)
\]  Note that
\[ \begin{bmatrix}
B&0\\
0&B^t
\end{bmatrix}\in \cS_D. \] We have
\[\begin{bmatrix}
A\circ B&0\\
0&A^t\circ B^t
\end{bmatrix}\in \cJ.\] Thus $A\circ B\in \cJ_0$. This proves Claim 1.

Moreover, in view of (1), it follows from Lemma \ref{L:Jordan} (ii) that $A^t\in \cJ_0$. Thus
$\cJ_0$ is invariant under the map $X\mapsto X^t$.

By \cite[Theorem 3]{Fong1982}, $\cJ_0$ is an  associative ideal of $\BH$. It follows immediately that $M_2(\cJ_0)$ is an  associative ideal of $\B(\cH^{(2)})$, where
\[M_2(\cJ_0) =  \left\{  \begin{bmatrix}
X&Y\\
Z&W
\end{bmatrix} : X,Y,Z,W\in \cJ_0 \right\}.\] Thus it remains to check that $\cJ=M_2(\cJ_0)\cap \cS_D$.

{\it Claim 2}. $M_2(\cJ_0)\cap \cS_D =\mathcal{E}$, where \[
  \mathcal{E}=\left\{  \begin{bmatrix}
A&E+E^t\\
F+F^t&A^t
\end{bmatrix} : A,E,F\in \cJ_0 \right\}.\]

Note that an operator $X$ lies in $\cJ_0$ if and only if $X^t\in\cJ_0$. Then, by Lemma \ref{L:Jordan} (i), the inclusion $\mathcal{E}\subset M_2(\cJ_0)\cap \cS_D$ is obvious.
Conversely, if $A,B,E,F\in \cJ_0$ and
$$\begin{bmatrix}
A&E\\
F&B
\end{bmatrix} \in \cS_D.$$ Then by Lemma \ref{L:Jordan} (i), $B=A^t$, $E=E^t$ and $F=F^t$.  It follows that $\mathcal{E}= M_2(\cJ_0)\cap \cS_D$.

Now it remains to check that $\cJ=\mathcal{E}$.

``$\cJ\subset \mathcal{E}$".
Choose an element $T\in \cJ$ and assume that
\[T=\begin{bmatrix}
A&E\\
F&A^t
\end{bmatrix},\] where $E=E^t$ and $F=F^t$. Thus $A\in \cJ_0$ and, by Lemma  \ref{L:Jordan} (ii), we have
\[\begin{bmatrix}
E&0\\
0&E
\end{bmatrix}\in \cJ, \ \ \ \begin{bmatrix}
F&0\\
0&F
\end{bmatrix} \in \cJ.\] This implies that $E,F\in \cJ_0$. So
\[ T=\begin{bmatrix}
A&E\\
F&A^t
\end{bmatrix}= \begin{bmatrix}
A&E/2+E^t/2\\
F/2+F^t/2&A^t
\end{bmatrix} \in \mathcal{E}.\]

``$\cJ\supset\mathcal{E}$".  Choose $A,E,F\in \cJ_0$. We shall prove that
\[T=\begin{bmatrix}
A&E+E^t\\
F+F^t&A^t
\end{bmatrix}\in \cJ.\]

If $X\in \cJ_0$, then one can see from the proof of Claim 1 that $X\oplus X^t\in \cJ$.
So we have
\[\begin{bmatrix}
A&0\\
0&A^t
\end{bmatrix}, \begin{bmatrix}
E&0\\
0&E^t
\end{bmatrix},  \begin{bmatrix}
F&0\\
0&F^t
\end{bmatrix}\in \cJ. \]
By Lemma  \ref{L:Jordan} (ii), we have $T\in \cJ$. This ends the proof.
\end{proof}

\begin{corollary}\label{C:universal}
 Let $C$ be a conjugation on $\cH$. Then $\BH$ is Jordan isomorphic to a Jordan subalgebra of $\SC$.
\end{corollary}

\begin{proof}
Set
$$D=\begin{bmatrix}
0&C\\
C&0
\end{bmatrix}.$$ Thus $D$ is a conjugation on $\cH^{(2)}$. By Corollary \ref{C:unitary}, $\SC$ and $\cS_D$ are unitarily equivalent.
Hence it suffices to prove that $\BH$ is Jordan isomorphic to a Jordan subalgebra of $\cS_D$.
Define
\begin{align*}
\phi: \BH&\longrightarrow \cS_D,\\
X& \longmapsto X\oplus X^t.\end{align*}
It is easy to see that $\phi$ is a linear isometry satisfying $\phi(X^*)=\phi(X)^*$ and
$\phi(X\circ Y)=\phi(X)\circ \phi(Y)$ for $X,Y\in\BH$. Thus $\phi$ induces a Jordan isomorphism between $\BH$ and $\phi(\BH)$.
\end{proof}

We denote by $\FH$ the set of all finite-rank operators on $\cH$, and by $\KH$ the set of all compact operators on $\cH$.

\begin{corollary}\label{C:minIdeal}
If $\mathcal{J}$ is a nontrivial Jordan ideal of $\SC$, then
$$[\SC\cap\FH] \subset \mathcal{J}\subset [\SC\cap\KH].$$
\end{corollary}

\begin{proof}
Note that each nontrivial associative ideal $\mathcal{I}$ of $\BH$ satisfies $\FH\subset\mathcal{I} \subset \KH$. The the desired result follows readily from Theorem \ref{T:Jordan}.
\end{proof}

\begin{corollary}
$\SC\cap\KH$ is the unique nontrivial norm-closed Jordan ideal of $\SC$.
\end{corollary}

\begin{remark}
The result of Theorem \ref{T:Jordan} still holds in the case that $\dim\cH<\infty$.
In fact, if $\dim\cH<\infty$, then one can  prove that neither $\BH$ nor $\SC$ has a nontrivial Jordan ideal. This shows that $\SC$ and $\BH$ has the same Jordan ideal structure for separable complex Hilbert space $\cH$.
\end{remark}

\subsection{Schatten $p$-classes}

%

The Schatten $p$-class of compact operators on $\cH$ is denoted by $\BPH$, $1\leq p<\infty$. It is well known that $\BPH$ is a Banach space under $p$-norm $\|\cdot\|_p$ and $\FH$ is dense in $\BPH$, where $\FH$ denotes the collection of finite-rank operators in $\BH$. Moreover, $\BPH$ is the  dual of $\B_q(\cH)$ for $1< p, q<\infty$ with $\frac{1}{p}+\frac{1}{q}=1$. The reader is refereed to \cite{Ringrose} or \cite{Schatten} for more details.
The aim of this subsection is to prove analogues of these facts in $\SC$.

For convenience, we denote $\B_0(\cH)=\KH$.
For $p\in\{0\}\cup [1,\infty)$, we denote $\SCP=\SC\cap\BPH$.
Still $\SCP$ is a Banach space under the norm $\|\cdot\|_p$. Note that $\|\cdot\|_0=\|\cdot\|$.

\begin{proposition}
Let $p\in\{0\}\cup [1,\infty)$. Then, given $T\in\SCP$ and $\eps>0$, there exists $F\in \SC\cap\FH$ such that $\|T-F\|_p<\eps$.
\end{proposition}

\begin{proof}
Choose an operator  $T\in\SCP$. Then $T\in\BPH$ and there exist finite-rank operators $\{F_n\}\subset\BH$ such that $\|T-F_n\|_p\rightarrow0$.
For each $n\geq 1$, note that $\frac{1}{2}(F_n+CF_n^*C)$ is of finite rank lying in $\SC$ and
\begin{align*}
\|T-\frac{1}{2}(F_n+CF_n^*C)\|_p&=\|\frac{1}{2}(T+CT^*C)-\frac{1}{2}(F_n+CF_n^*C)\|_p\\
&\leq \|\frac{1}{2}(T-F_n)\|_p+\|\frac{1}{2}C(T-F_n)^*C\|_p \\
&=\|T-F_n\|_p.
\end{align*}
Thus $T$ can be approximated in $p$-norm $\|\cdot\|_p$ by finite-rank operators in $\SC$.
\end{proof}

%
%
%

%
%
%

%

It is found in \cite{ZhuOAM2017} that $\SC$ has a topological complement. In fact, the set of skew-symmetric operators relative to $C$
\[\SSC:=\{X\in\BH: CXC=-X^*\}\] is a topological complement of $\SC$; that is, $\SC+\SSC=\BH$ and $\SC\cap\SSC=\{0\}$.
$\SSC$ is just the symplectic type Cartan factor (see \cite{ChuCH}).
It was shown in \cite{ZhuOAM2017} that $\SSC$ is  Roberts orthogonal to $\SC$. Recall that two operators $A,B\in\BH$ are said to be {\it Roberts orthogonal}, if $\|A-\lambda B\|= \|A+\lambda B\|$ for all complex numbers $\lambda$. 

\begin{lemma}
$\SC+\KH$ is a proper, norm-closed Jordan subalgebra of $\BH$.
\end{lemma}

\begin{proof}
Clearly, $\SC+\KH$ is a linear subspace of $\BH$ and closed under the Jordan product. Also each Fredholm operator in $\SC+\KH$ has an index $0$. Hence it suffices to prove that $\SC+\KH$ is norm-closed.

Assume that $\{A_n+K_n\}$ is a Cauchy sequence, where $\{A_n\}\subset\SC$ and $\{K_n\}\subset\KH$.
Set $K_n^+=\frac{1}{2}(K_n+CK_n^*C)$ and $K_n^-=\frac{1}{2}(K_n-CK_n^*C)$. Then $K_n^+\in\SC$ and $K_n^-\in\SSC\cap\KH$.
By \cite[Theorem 2.1]{ZhuOAM2017}, we have
\[ \|K_n^--K_m^-\|\leq \|(A_n+K_n)-(A_m+K_m)\|.\] Thus $\{K_n^-\}$ is a Cauchy sequence and converges to a compact operator $K^-$.
It follows that $\{A_n+K_n^+\}$ is a Cauchy sequence in $\SC$ and converges to an operator $A\in\SC$. Therefore we conclude that $A_n+K_n\rightarrow A+K^-\in\SC+\KH$.
\end{proof}


%

The aim of the rest of this subsection is to prove the following result which exhibits the dual relation among $\SC$ and $\SCP $ $(p\in\{0\}\cup[1,\infty)$.

\begin{proposition}\label{P:SchattenP}
\begin{enumerate}
\item[(i)] $(\SCY,\|\cdot\|_1)$ is isometrically isomorphic to the dual
of $(\SCO,\|\cdot\|)$.
\item[(ii)] $\SC$ is isometrically isomorphic to the dual
of $(\SCY,\|\cdot\|_1)$.
\item[(iii)] $(\SCq,\|\cdot\|_q)$ is isometrically isomorphic to the dual
of $(\SCP,\|\cdot\|_p)$, where $1<p,q<\infty$ and $\frac{1}{p}+\frac{1}{ q}=1$.
\end{enumerate}
\end{proposition}

In order to prove the preceding proposition, we need to make some preparation.

Given a complex matrix $A$, we denote by $A^{\textup{T}}$ the transpose of $A$.

\begin{lemma}\label{L:SchattenP}
Let $n$ be a positive integer and $M_n(\bC)$ be the collection of $n\times n$ complex matrices.
If $A,B\in M_n(\bC)$ with $A=A^{\textup{T}}$ and $B=-B^{\textup{T}}$, then $\tr(AB)=0$, where $\tr(\cdot)$ is the trace function.
\end{lemma}

\begin{proof}
Note that $\tr(AB)=\tr(AB)^{\textup{T}}=\tr(B^{\textup{T}} A^{\textup{T}})=-\tr(B A)=-\tr(AB)$. So $\tr(AB)=0$.
\end{proof}

\begin{corollary}\label{C:SchattenP}
Let $C$ be a conjugation on $\cH$. Assume that $A\in\SC$ and $B\in\SSC$.
If (i) $A\in \B_1(\cH)$, or (ii) $A\in\BPH$ and $B\in \B_q(\cH)$, where $1<p,q<\infty$ and $\frac{1}{p}+\frac{1}{q}=1$, then $\tr(AB)=0$.
\end{corollary}

\begin{proof}
We just give the proof in the case (i). The proof for the case (ii) is similar.

Since $C$ is a conjugation on $\cH$, by \cite[Lemma 2.11]{Gar14}, there exists an orthonormal basis $\{e_n\}$ such that $Ce_n=e_n$ for all $n$.
For each $n\geq 1$, denote by $P_n$ the projection of $\cH$ onto $\vee\{e_i: 1\leq i\leq n\}$.

Note that $P_n\rightarrow I$ in the strong operator topology. It follows that $\lim_n\|P_nAP_n-A\|_1=0$ and furthermore
$$\|P_nAP_nB-AB\|_1\leq \|P_nAP_n-A\|_1\cdot\|B\|\rightarrow 0$$ as $n\rightarrow \infty$. Thus $\tr(AB)=\lim_n\tr(P_nAP_nB)$.
It suffices to prove that $\tr(P_nAP_nB)=0$ for all $n$.

For each $n$, assume that
\[A=\begin{bmatrix}
A_n&*\\
*&*
\end{bmatrix}\begin{matrix}
\ran P_n \\ \ran (I-P_n)
\end{matrix},\ \ \ B=\begin{bmatrix}
B_n&*\\
*&*
\end{bmatrix}\begin{matrix}
\ran P_n \\ \ran (I-P_n)
\end{matrix}.\] It follows that $\tr(P_nAP_nB)=\tr(A_nB_n)$.
For $1\leq i,j\leq n$, note that
\begin{align*}
\la A_n e_i,e_j\ra&=\la A e_i,e_j\ra=\la A e_i, C e_j\ra\\
&=\la e_i, A^* C e_j\ra=\la e_i, C A e_j\ra\\
&=\la A e_j, C e_i \ra=\la A e_j, e_i \ra\\
&=\la A_n e_j, e_i \ra
\end{align*} and similarly that $\la B_n e_i,e_j\ra=-\la B_n e_i,e_j\ra$. Thus, relative to $\{e_1,e_2,\cdots,e_n\}$,
$A_n$ admits a symmetric matrix representation and $B_n$ admits a skew symmetric matrix representation (that is, $R=-R^T$).
By Lemma \ref{L:SchattenP}, $\tr(P_nAP_nB)=\tr(A_nB_n)=0$. Therefore we conclude that $\tr(AB)=0$.
\end{proof}

Given a Banach space $\mathcal{X}$, we let $\mathcal{X}'$ denote its dual.

\begin{proof}[Proof of Proposition \ref{P:SchattenP}]
(i) For $K\in \SCY$, define $\phi_K: \SCO \rightarrow \bC$ as
\[\phi_K(X)=\tr(XK), \ \ \forall X\in \SCO. \]
Clearly, $|\phi_K(X)|=|\tr(XK)|\leq\|K\|_1\cdot\|X\|$. Thus $\phi_K\in(\SCO)'$.

It suffices to prove that the map $\Phi: K\longmapsto \phi_K $ is an isometric isomorphism of $\SCY$ onto  $(\SCO)'$.
Clearly,  $\Phi$ is linear. It remains to check that $\Phi$ is isometric and surjective.

{\it Step 1.} $\Phi$ is isometric.

Fix a $K\in \SCY$. $\phi_K$ can be extended to the linear functional $\widetilde{\phi}_K$ on $\KH$ defined by
\[\widetilde{\phi}_K(X)=\tr(XK), \ \  \forall X\in  \KH. \]
Then, by \cite[Theorem 19.1]{Conway00}, $\|\phi_K\|\leq \|\widetilde{\phi}_K\|=\|K\|_1$.

For any $Y\in\KH$, denote $Y_1=\frac{1}{2}(Y+CY^*C)$ and $Y_2=\frac{1}{2}(Y-CY^*C)$. Note that $\|Y_1\|\leq\|Y\|$, $Y_1\in\SCO$ and $Y_2\in\SSC\cap\KH$. By Corollary \ref{C:SchattenP},
\begin{align*}
|\widetilde{\phi}_K(Y)|&=|\tr(KY_1+KY_2)|=|\tr(KY_1)|\\
&=|\phi_K(Y_1)|\leq \|\phi_K\|\cdot \|Y_1\|\\
& \leq \|\phi_K\|\cdot \|Y\|.
\end{align*} Since $Y\in\KH$ was arbitrary operator in $\KH$, we deduce that $\|\widetilde{\phi}_K\|\leq \|\phi_K\|$. Furthermore, we obtain
$ \|\phi_K\|=\|\widetilde{\phi}_K\|=\|K\|_1$. This shows that  $\Phi$ is isometric.

{\it Step 2.} $\Phi$ is surjective.

Assume that $l$ is a bounded linear functional on $\SCO$. Since $\SCO $ is a closed subspace of $\KH$, $l$ admits an extension $\tilde{l}$ to $\KH$. Then, by \cite[Theorem 19.1]{Conway00}, there exists an operator $K\in\B_1(\cH) $ such that $\tilde{l}(X)=\tr(XK)$ for $X\in\KH$.
Denote $K_1=\frac{1}{2}(K+CK^*C)$ and $K_2=\frac{1}{2}(K-CK^*C)$. Then $K_1\in\SCY$ and $K_2\in\SSC\cap\B_1(\cH)$.
For each $X\in\SCO$, we have \begin{align*}
l(X)=\tilde{l}(X)=\tr(XK)=\tr(XK_1)+\tr(XK_2)=\tr(XK_1)=\phi_{K_1}(X),
\end{align*}
which means that $l=\phi_{K_1}$. Thus we conclude that $\Phi$ is surjective.


(ii) For $K\in \SC$, define $\psi_K: \SCY\rightarrow \bC$ as
\[\psi_K(X)=\tr(XK), \ \ X\in \SCY. \]
Clearly, $|\psi_K(X)|=|\tr(XK)|\leq\|K\|\cdot\|X\|_1$. Thus $\psi_K\in(\SCY)'$.

It suffices to prove that the map $\Psi: K\longmapsto \psi_K $ is an isometric isomorphism of $\SC$ onto  $(\SCY)'$.
Clearly,  $\Psi$ is linear. It remains to check that $\Psi$ is isometric and surjective.

{\it Step 1.} $\Psi$ is isometric.

Fix a $K\in \SC $. Obviously, $\psi_K$ can be extended to the linear functional $\widetilde{\psi}_K$ on $\B_1(\cH)$ defined by
\[\widetilde{\psi}_K(X)=\tr(XK), \ \ X\in \B_1(\cH). \]
Then, by \cite[Theorem 19.2]{Conway00}, $\|\psi_K\|\leq \|\widetilde{\psi}_K\|=\|K\|$.

For any $Y\in\B_1(\cH)$, denote $Y_1=\frac{1}{2}(Y+CY^*C)$ and $Y_2=\frac{1}{2}(Y-CY^*C)$. Note that $\|Y_1\|_1\leq\|Y\|_1$, $Y_1\in\SCY$ and $Y_2\in\SSC\cap\B_1(\cH)$. By Corollary \ref{C:SchattenP},
\begin{align*}
|\widetilde{\psi}_K(Y)|&=|\tr(KY_1+KY_2)|=|\tr(KY_1)|\\
&=|\psi_K(Y_1)|\leq \|\psi_K\|\cdot \|Y_1\|_1\\
& \leq \|\psi_K\|\cdot \|Y\|_1.
\end{align*} Since $Y $ was arbitrary operator in $\B_1(\cH)$, we deduce that $\|\widetilde{\psi}_K\|\leq \|\psi_K\|$. Furthermore, we obtain
$ \|\psi_K\|=\|\widetilde{\psi}_K\|=\|K\|$. This shows that  $\Psi$ is isometric.

{\it Step 2.} $\Psi$ is surjective.

Assume that $l$ is a bounded linear functional on $\SCY$. Since $\SCY$ is a closed subspace of $\B_1(\cH)$, $l$ admits an extension $\tilde{l}$ to $\B_1(\cH)$. Then, by \cite[Theorem 19.2]{Conway00}, there exists an operator $K\in\BH $ such that $\tilde{l}(X)=\tr(XK)$ for $X\in\B_1(\cH)$.
Denote $K_1=\frac{1}{2}(K+CK^*C)$ and $K_2=\frac{1}{2}(K-CK^*C)$. Then $K_1\in\SC $ and $K_2\in\SSC $.
For each $X\in\SCY$, we have \begin{align*}
l(X)=\tilde{l}(X)=\tr(XK)=\tr(XK_1)+\tr(XK_2)=\tr(XK_1)=\psi_{K_1}(X),
\end{align*}
which means that $l=\psi_{K_1}$. Thus we conclude that $\Psi$ is surjective.

(iii) The proof follows similar lines as those of (i) and (ii), and is omitted.
\end{proof}

%

%



\subsection{$C^*$-algebras contained in $\SC$}
This subsection aims to characterize the $C^*$-algebras contained in $\SC$ and the $C^*$-algebras generated by $\SC$.
This helps to develop operator theory in $\SC$.

For $T\in\BH$, we denote by $J^*(T)$ the Jordan operator algebra generated by $T$, $T^*$ and the identity operator $I$, and denote by $W^*(T)$
the von Neumann algebra generated by $T$.

The first result shows that $\SC$ contains no noncommutative $C^*$-algebra.

\begin{proposition}\label{P:funcCalcu}
Let $C$ be a conjugation on $\cH$ and $T\in\SC$. Then the following are equivalent:
 \begin{enumerate}
 \item[(i)] $C^*(T)\subset \SC$;
 \item[(ii)] $W^*(T)\subset \SC$;
 \item[(iii)] $|T|\in\SC$;
\item[(iv)] $T$ is normal;
  \item[(v)] $C^*(T)=J^*(T)$.
 \end{enumerate}
\end{proposition}

\begin{proof}
Since $\SC$ is closed in the weak operator topology, the implication  (i)$\Longrightarrow$(ii) is obvious.

(ii)$\Longrightarrow$(iii).  This is obvious.

(iii)$\Longrightarrow$(iv).
Since $|T|\in \SC$, we have $C|T|C=|T|$ and hence $C|T|^2C=|T|^2=T^*T$. On the other hand,
\[C|T|^2C=CT^*TC=(CT^*C)(CTC)=TT^*.\] Thus $T^*T=TT^*$.

(iv)$\Longrightarrow$(v). Since $C^*(T)$ is an associative algebra and hence a Jordan algebra containing $T,T^*$, it follows that $J^*(T)\subset C^*(T)$. Note that $T$ is normal. Thus $C^*(T)$ is the closed linear span of $\{T^m{T^*}^n: m,n\geq 0\}$.
Noting that \[T=T\circ I, \ \ T^2=T\circ T,\ \  T^3=(T\circ T)\circ T, \cdots,\] we have $T^m,{T^*}^n\in J^*(T)$. Also, one can check
\[T^m{T^*}^n=T^m\circ {T^*}^n,\] since $T^m{T^*}^n={T^*}^n T^m$. This shows that $C^*(T)\subset J^*(T)$. We conclude that $C^*(T)= J^*(T)$.


(v)$\Longrightarrow$(i). Note that $T,T^*\in\SC$. Thus $J^*(T)$ is a Jordan subalgebra of $\SC$. It follows that $C^*(T)\subset \SC$.
\end{proof}

\begin{remark}\label{R:funcCalcu}
The Jordan product $\circ$ on $\SC$ is not associative. In fact, choose a non-normal $T\in\SC$. Thus $T^*\in\SC$. By Proposition \ref{P:funcCalcu}, $|T|\notin\SC$ and hence $|T|^2=T^*T\notin\SC$.
\end{remark}

 By Proposition \ref{P:funcCalcu}, a good functional calculus is permitted in $\SC$. This will have many useful corollaries.

\begin{corollary}\label{C:funcCalcu}
Let $C$ be a conjugation on $\cH$ and $\cA$ be a JC$^*$-subalgebra of $\SC$. If $T\in\cA$, then
$C^*(T)\subset \cA$ if and only if $T$ is normal.
\end{corollary}

\begin{corollary}
If $\cA$ is a $C^*$-subalgebra of $\BH$ with $\cA\subset\SC$, then $\cA\subset W^*(N)$ for some normal operator $N\in\SC$.
\end{corollary}

\begin{corollary}
The selfadjoint elements of $\SC$ with finite spectra are norm dense in the selfadjoint elements of $\SC$.
\end{corollary}

Since $\SC\subset\BH$ and $\SC$ is not an associative algebra, it is natural to determine the associative algebra generated by $\SC$.

\begin{proposition}\label{T:CTEnelope}
$\BH$ is the $C^*$-algebra generated by $\SC$.
\end{proposition}

\begin{proof}
Set
$$D=\begin{bmatrix}
0&C\\
C&0
\end{bmatrix}.$$ Thus $D$ is a conjugation on $\cH^{(2)}$.  It suffices to prove that each operator $T$ on $\cH^{(2)}$ lies in the $C^*$-algebra $C^*(\cS_D)$ generated by $\cS_D$.

Note that \[\begin{bmatrix}
0&I\\
0&0
\end{bmatrix}\in\cS_D.\] It follows immediately that $M_2(\bC I)\subset C^*(\cS_D)$. For any $X\in\BH$,
$X\oplus CX^*C\in\cS_D$. One can check that
\[\begin{bmatrix}
X&0\\
0&CX^*C
\end{bmatrix}\circ \begin{bmatrix}
I&0\\
0&0
\end{bmatrix}=\begin{bmatrix}
X&0\\
0&0
\end{bmatrix}\in C^*(\cS_D) \] and
\[\begin{bmatrix}
X&0\\
0&0
\end{bmatrix}\begin{bmatrix}
0&I\\
0&0
\end{bmatrix}=\begin{bmatrix}
0&X\\
0&0
\end{bmatrix}\in C^*(\cS_D).\] Since $C^*(\cS_D)$ is selfadjoint, it follows that
\[\begin{bmatrix}
0&0 \\
X&0
\end{bmatrix}\in C^*(\cS_D)\] and
\[\begin{bmatrix}
0&0\\
X&0
\end{bmatrix}\begin{bmatrix}
0&I\\
0&0
\end{bmatrix}=\begin{bmatrix}
0&0\\
0&X
\end{bmatrix}\in  C^*(\cS_D). \] Therefore we conclude that $\B(\cH^{(2)})=C^*(\cS_D)$.
\end{proof}

\subsection{Jordan automorphisms}\label{S:JordanAuto}

A map $\varphi:\SC\rightarrow\SC$ is called a {\it Jordan automorphism} of $\SC$ if
$\varphi$ is linear, bijective and
$$\varphi(X^*)=\varphi(X)^*, \ \ \varphi(X\circ Y)=\varphi(X)\circ\varphi(Y)$$
for all $X,Y\in\SC$.

The main result of this subsection is the following theorem which determines Jordan automorphisms of $\SC$.

\begin{theorem}\label{T:JordanIso}
A map $\varphi:\SC\rightarrow\SC$ is a Jordan automorphism of $\SC$ if and only if
there exists a unitary operator $V\in\BH$ with $CV=VC$ such that $\varphi(X)=VXV^*$ for all $X\in\SC$.
\end{theorem}

%

To give the proof of Theorem \ref{T:JordanIso}, we make some preparation.

\begin{lemma}\label{L:CD}
Let $C$ and $D$ be conjugations on $\cH$. Then the following are equivalent.
\begin{enumerate}
\item[(i)] $\mathcal{S}_C=\mathcal{S}_D$
\item[(ii)] $\mathcal{S}_C\subset \mathcal{S}_D$
\item[(iii)] $\mathcal{S}_C\supset \mathcal{S}_D$
\item[(iv)] $C=\alpha D$ for some unimodular $\alpha\in\bC.$
\end{enumerate}
\end{lemma}

\begin{proof}
It suffices to prove (ii)$\Longrightarrow$(iv).

``(ii)$\Longrightarrow$(iv)". Assume that $\{e_i\}_{i\geq1}$ is an orthonormal basis (\onb) of $\cH$
such that $Ce_i=e_i$ for all $i$. Easy to see $e_i\otimes e_i\in\SC$ for $i\geq 1$.
Thus $e_i\otimes e_i\in \mathcal{S}_D$.
That is, $D(e_i\otimes e_i)D=e_i\otimes e_i$.
Since $D(e_i\otimes e_i)D=(De_i)\otimes (De_i)$, there exists unimodular $\alpha_i\in\bC$ such that $De_i=\alpha_i e_i$.

Now it remains to show that $\alpha_i=\alpha_1$ for all $i\geq 2$.
Assume that $i\geq 2$. Note that $X=e_1\otimes e_i+e_i\otimes e_1\in\SC$.
Thus $DXD=X^*=X$. In particular,
$DXDe_i=Xe_i$. Since $Xe_i=e_1$ and $DXDe_i=DX(\alpha_i e_i)=\alpha_1\overline{\alpha_i}e_1,$ it follows that $e_1=\alpha_1\overline{\alpha_i}e_1$, that is, $\alpha_1=\alpha_i$.
\end{proof}

\begin{lemma}\label{L:unitaryImple}
Let $C$ be a conjugation on $\cH$ and $U$ be a unitary operator on $\cH$. For $X\in\BH$, $\psi(X)=UXU^*$. Then
$\psi(\SC)=\SC$ if and only if $C$ commutes with $\alpha U$ for some unimodular $\alpha\in\bC.$
\end{lemma}

\begin{proof}
The sufficiency is obvious. We need only prove the necessity.

``$\Longrightarrow$". Since $U$ is unitary, it is easy to check that $\psi(\SC)=\mathcal{S}_{UCU^*}$. So $\SC=\mathcal{S}_{UCU^*}$.
By Lemma \ref{L:CD}, there exists unimodular $\beta\in\bC$ such that $UCU^*=\beta C$. Thus $UC=\beta CU$.
Assume that $\alpha\in\bC$ satisfies $\alpha^2=\overline{\beta}$.
Then $$(\alpha U)C=\alpha (UC) =(\alpha\beta) CU=\overline{\alpha } CU= C(\alpha U).$$
This ends the proof.
\end{proof}

If $C$ is a conjugation on $\cH$, then we denote by $\FC$ the collection of all finite-rank operators in $\SC$.

\begin{lemma}\label{L:JIBasic}
Let $\varphi$ be a Jordan automorphism of $\SC$.
Then
\begin{enumerate}
\item[(i)] $\varphi(\FC)=\FC$;
\item[(ii)] $\varphi(X^2)=\varphi(X)^2$ for all $X\in \SC$;
\item[(iii)] if $X\in\SC$ is positive, then so is $\varphi(X)$;
\item[(iv)] if $X\in\SC$ is a projection of rank one, then so is $\varphi(X)$;
\item[(v)] if $X\in\SC$ is selfadjoint, then $\|\varphi(X)\|=\|X\|$.
\end{enumerate}
\end{lemma}

\begin{proof}
(i) By Corollary \ref{C:minIdeal}, $\FC$ is the smallest nonzero Jordan ideal of $\SC$. It follows immediately that $\varphi(\FC)=\FC$.

(ii) For $X\in \BH$, note that $X\circ X=X^2$. Then the result follows readily.

(iii) Since $X$ is positive, it follows that $X=Y^2$ for some positive $Y\in\BH$.
In view of Proposition \ref{P:funcCalcu}, it can be required that $Y\in\SC$. Then $\varphi(X)=\varphi(Y^2)=\varphi(Y)^2$.
Note that $\varphi(Y)=\varphi(Y^*)=\varphi(Y)^*$. Thus $\varphi(X)$ is positive.

(iv) From (ii) and (iii), one can see that $\varphi(X)$ is a projection. It remains to check that $\rank\varphi(X)=1$.
Note that $\varphi$ maps positive operators to positive operators. Then $\varphi$ maps minimal projections to minimal projections. So the desired result follows readily.

(v) Denote by $C^*(X)$ the $C^*$-subalgebra of $\BH$ generated by $X$ and the identity $I$.
Then $C^*(X)$ is commutative and $C^*(X)\subset \SC$.
One can check that $\varphi|_{C^*(X)}$ is a faithful representation of $C^*(X)$, and hence is isometric.
\end{proof}

\begin{proposition}\label{P:unitImple}
Let $\varphi$ be a Jordan automorphism of $\SC$ and $\{e_i\}_{i\geq1}$ be an \onb~ of $\cH$
such that $Ce_i=e_i$ for all $i$.  Then there exists a unitary operator $U\in\BH$ such that
$\varphi(X)=UXU^*$ for all $X\in\Theta$, where
$$\Theta=\{X\in\SC: \exists n\geq 1 ~\textup{such that}~ \la Xe_i,e_j\ra= 0~ \textup{whenever}~ i+j\geq n\}$$
\end{proposition}

\begin{proof}
For $i\geq 1$, set $P_i=e_i\otimes e_i$ and $Q_i=\varphi(P_i)$.
Clearly, $P_i$ is a projection of rank one.
In view of Lemma \ref{L:JIBasic}, $Q_i$ is a projection of rank one.
Then there exists a unit vector $f_i$ such that $Q_i=f_i\otimes f_i$.
Using Lemma \ref{L:JIBasic} again, one can see that $\{f_i\}_{i\geq1}$ is an \onb~ of $\cH$.

{\it Claim.} There exists unimodular numbers $\{\alpha_i\}_{i\geq 1}$ such that $\varphi(e_i\otimes e_j+e_j\otimes e_i)=(\alpha_if_i)\otimes (\alpha_jf_j)+(\alpha_jf_j)\otimes (\alpha_if_i)$.

For $i,j\geq 1$ with $i<j$, denote $E_{i,j}=e_i\otimes e_j+e_j\otimes e_i$. Then $E_{i,j}$ lies in $\SC$ and is selfadjoint. Denote
$F_{i,j}=\varphi(E_{i,j})$. One can see that $F_{i,j}$ is selfadjoint and, by Lemma \ref{L:JIBasic} (v), $\|F_{i,j}\|=1$.

Now fix $i,j\geq 1$ with $i< j$.
Note that
$Q_k\circ F_{i,j}=\varphi(P_k\circ E_{i,j})=0$ whenever $k\notin\{i,j\}$.
So $$F_{i,j}=a_{i,j} f_i\otimes f_i+b_{i,j} f_j\otimes f_i+\overline{b_{i,j}} f_i\otimes f_j+c_{i,j} f_j\otimes f_j$$ for some $a_{i,j},b_{i,j},c_{i,j}\in\bC$.
Thus $$Q_i\circ F_{i,j}=a_{i,j} f_i\otimes f_i+\frac{b_{i,j}}{2} f_j\otimes f_i+\frac{\overline{b_{i,j}}}{2} f_i\otimes f_j.$$
On the other hand, note that $$Q_i\circ F_{i,j}=\varphi(P_i\circ E_{i,j})=\frac{1}{2}\varphi(E_{i,j})=\frac{1}{2}F_{i,j}.$$
It follows that $a_{i,j}=0=c_{i,j}$ and $|b_{i,j}|=1$. Hence
$$F_{i,j}=b_{i,j} f_j\otimes f_i+\overline{b_{i,j}} f_i\otimes f_j.$$

Set $\alpha_{1}=1$ and $\alpha_i=\prod_{k=1}^{i-1}b_{k,k+1}$ for $i\geq 2$.
Set $g_i=\alpha_if_i$ $i\geq 1$ and $G_{i,j}=g_i\otimes g_{j}+g_{j}\otimes g_i$ for $i>j\geq 1$.

%

For $i\geq 1$,
\begin{align*}
  G_{i,i+1}&=g_i\otimes g_{i+1}+g_{i+1}\otimes g_i\\
  &=(\alpha_if_i)\otimes (\alpha_{i+1}f_{i+1})+(\alpha_{i+1}f_{i+1})\otimes (\alpha_if_i)\\
  &=\overline{b_{i,i+1}}(f_i\otimes f_{i+1})+b_{i,i+1} (f_{i+1}\otimes f_i)=\varphi(E_{i,i+1}).
\end{align*}
That is,
\[
\varphi(E_{i,i+1})=G_{i,i+1}. \eqno(2)
\]

For $i\geq 1$ and $k\geq 1$, we check that
$$E_{i,i+k}\circ E_{i+k,i+k+1}=\frac{1}{2}E_{i,i+k+1}\ \  \textup{and}\ \  G_{i,i+k}\circ G_{i+k,i+k+1}=\frac{1}{2}G_{i,i+k+1}.$$
In view of (2), we have  $\varphi(E_{i,i+k})=G_{i,i+k}$. This proves Claim.

Clearly, $\{g_i\}_{i\geq 1}$ is an \onb~ of $\cH$. Define unitary $U\in\BH$ as $Ue_i=g_i$ for $i\geq 1$.
Then, for $i,k\geq 1$, $$\varphi(P_{i})=Q_i=f_i\otimes f_i=g_i\otimes g_i=(Ue_i)\otimes (Ue_i)=UP_iU^*$$ and
$$\varphi(E_{i,i+k})=G_{i,i+k}=g_i\otimes g_{i+k}+g_{i+k}\otimes g_{i}=UE_{i,i+k}U^*.$$

Clearly, $\Theta$ equals the linear span of $\{ P_i, E_{i,j}: j>i\geq 1\}$. So $\varphi(X)=UXU^*$ for
all $X\in \Theta$.
\end{proof}

\begin{corollary}
If $\varphi$ is a Jordan automorphism of $\SC$, then $\|\varphi(X)\|=\|X\|$ for $X\in \FC$.
\end{corollary}

\begin{proof}
Choose an $X\in\FC$. Set $M=\overline{\ran X+\ran X^*}$. Then $M$ reduces both $C$ and $X$.
Assume that $n=\dim M$. Then we can choose an \onb~ $\{e_i\}_{i\geq 1}$ of $\cH$ such that $M=\vee\{e_i: 1\leq i\leq n\}$ and $Ce_i=e_i$ for all $i\geq 1$.

By Proposition \ref{P:unitImple}, there exists a unitary operator $U\in\BH$ such that
$\varphi(X)=UXU^*$ for all $X\in\Theta$, where
$$\Theta=\{Y\in\SC: \exists n\geq 1 ~\textup{such that}~ \la Ye_i,e_j\ra= 0~ \textup{whenever}~ i+j\geq n\}.$$
Clearly, $X\in\Theta$. Thus $\|\varphi(X)\|=\|X\|$. This completes the proof.
\end{proof}

\begin{proof}[Proof of Theorem \ref{T:JordanIso}]
First we choose an \onb~ $\{e_i\}_{i\geq 1}$ of $\cH$ such that $Ce_i=e_i$ for all $i\geq 1$.

By Proposition \ref{P:unitImple}, there exists a unitary operator $U\in\BH$ such that
$\varphi(X)=UXU^*$ for all $X\in\Theta$, where
$$\Theta=\{Y\in\SC: \exists n\geq 1 ~\textup{such that}~ \la Ye_i,e_j\ra= 0~ \textup{whenever}~ i+j\geq n\}.$$

{\it Claim.} $\varphi(X)=UXU^*$ for $X\in\FC$.

Arbitrarily choose an $X\in\FC$. For each $n\geq 1$, denote by $R_n$ the projection of $\cH$ onto $\vee\{e_i: 1\leq i\leq n\}$. Then $R_nXR_n\in\Theta$ and $R_nXR_n\rightarrow X$ in norm.
Since $\varphi|_{\FC}$ is isometric, it follows that \begin{align*}
\varphi(X)=\lim_n\varphi(R_nXR_n)
=\lim_nU(R_nXR_n)U^*
=UXU^*.
\end{align*} This proves the claim.

For $i\geq 1$, denote $g_i=Ue_i$ and $Q_i=g_i\otimes g_i$. Then $\{g_i\}$ is an \onb~ of $\cH$ and
$\varphi(P_i)=UP_iU^*=U(e_i\otimes e_i)U^*=Q_i$.

Now we shall prove that $\varphi(Y)=UYU^*$ for $Y\in\SC$.


Fix an operator $Y\in\SC$.
For each $i\geq 1$, since $Y\circ P_i\in\FC$, it follows that
\begin{align*}
\varphi(Y)\circ Q_{i}&=\varphi(Y)\circ \varphi(P_{i}) =\varphi(Y\circ P_{i})\\
&=U(Y\circ P_{i})U^* =(UYU^*)\circ (UP_{i}U^*)=(UYU^*)\circ Q_{i}.
\end{align*}
Then $\la \varphi(Y) g_i, g_j\ra=\la UYU^* g_i, g_j\ra$ for all $i,j$ with $i\ne j$.
On the other hand, for each $i\geq 1$,
note that
\begin{align*}
Q_{i} \varphi(Y) Q_{i}& =2Q_i \circ \big(Q_i\circ \varphi(Y)\big)-Q_{i}^2 \circ \varphi(Y) \\
&=2\varphi(P_i) \circ \Big(\varphi(P_i)\circ \varphi(Y)\Big)-\varphi(P_{i}^2) \circ \varphi(Y)\\
&=\varphi\Big(2 P_i \circ \big(P_i\circ Y\big)-P_{i}^2 \circ Y\Big)\\
&=U\Big(2 P_i \circ \big(P_i\circ Y\big)-P_{i}^2 \circ Y\Big) U^*\\
&=U( P_i  YP_i) U^*=Q_i(UY U^*)Q_i,
\end{align*} which implies that $\la \varphi(Y)g_i,g_i\ra=\la UYU^* g_i, g_i\ra$. Therefore we conclude that
$ \varphi(Y)= UYU^* $.

By Lemma \ref{L:unitaryImple}, there exists unimodular $\alpha\in\bC$ such that $(\alpha U)C=C(\alpha U)$.
Set $V=\alpha U$. Then $V$ is unitary and one can see that $\varphi(Y)=VYV^*$ for all $Y\in\SC$.
\end{proof}

\section{Jordan multiplication operators}\label{S:JordanMult}

The Jordan product $\circ$ naturally induces a class of multiplication operators on $\SC$.
For $T\in\SC$, define $J_T \in\B(\SC)$ as $J_T:X \longmapsto T\circ X.$ Thus $J_T$ is a bounded linear operator on $\SC$.

For $T\in\SC$, $J_T$ is closely related to the Rosenblum operator induced by $T$.
For $A,B\in\BH$, the Rosenblum operator $\tau_{A,B}$ on $\BH$ is defined by
\[\tau_{A,B}(X)=AX-XB,\ \ \ \    \forall X\in\BH.\]
Thus $J_T$ is the restriction of the Rosenblum operator $\frac{1}{2}\tau_{T,-T}$ to $\SC$.
Rosenblum operators, which arose in the study of operator equations, were first systematically studied by M. Rosenblum in \cite{Rosenblum}.

We wish to determine the spectrum of $J_T$ and its different parts for $T\in\SC$, since the spectrum of a Rosenblum operator has been clearly described (see \cite[Chapter 4]{Herr89} or \cite{Lumer,Rosenblum}).

Let $A$ be a bounded linear operator acting on some Banach space. Denote by $\ker A$ and $\ran A$ the kernel of $A$ and the range of $A$ respectively.
We let $\sigma_p(A)$, $\sigma_\pi(A)$ and $\sigma_\delta(A)$ denote respectively the point spectrum of $A$, the approximate point spectrum of $A$ and the approximate defect spectrum of $A$. Thus
$$\sigma_\pi(A)=\{z\in\bC: A-z ~\textup{is not bounded below}\}$$
and
$$\sigma_\delta(A)=\{z\in\bC: A-z ~\textup{is not surjective}\}.$$
We let $\sigma_l(A)$ and $\sigma_r(A)$ denote respectively the left spectrum of $A$ and the right spectrum of $A$.
That is, $$\sigma_l(A)=\{z\in\bC: A-z ~\textup{does not have a left inverse}\}$$
and
$$\sigma_r(A)=\{z\in\bC: A-z ~\textup{does not have a right inverse}\}.$$

The main result of this section is the following theorem.

\begin{theorem}\label{T:JMultiply}
Let $T\in\SC$. Then
\begin{enumerate}
\item[(i)] $\|J_T\|=\|T\|$;
\item[(ii)] $\sigma(J_T)=\sigma_r(J_T)=\sigma_\delta(J_T)=\sigma_l(J_T)=\sigma_\pi(J_T)=\frac{1}{2}[\sigma(T)+\sigma(T)]$.
\end{enumerate}
\end{theorem}

To give the proof of Theorem \ref{T:JMultiply}, we first make some preparation.

For the reader's convenience, we write down some elementary facts.

\begin{lemma}\label{L:computation}
Let $e,f\in\cH$ and $X=e\otimes f$. If $A\in\BH$ and $C$ is a conjugation on $\cH$, then
\begin{enumerate}
\item[(i)]$AX= (Ae)\otimes f$,
\item[(ii)] $X A=e\otimes (A^*f)$, and
\item[(iii)] $CXC= (Ce) \otimes (Cf)$.
\end{enumerate}
\end{lemma}

\begin{lemma}\label{L:CSeg}
Let $e,f\in\cH$ with $\|e\|=\|f\|=1$. Set $X=e\otimes f + (Cf) \otimes (Ce)$.
Then $X\in\SC$ and $1\leq \|X\|\leq \|X\|_p\leq 2$ for all $p\in[1,\infty)$.
\end{lemma}

\begin{proof}
It is easy to check that $CXC=X^*$ and $\|X\|_p\leq 2$.

On the other hand, compute to see
\begin{align*}
  \|X\|&\geq |\la Xf, e \ra|=|1+ \la f,Ce\ra\cdot \la Cf,e\ra |\\
  &=|1+ \la f,Ce\ra\cdot \la Ce,f\ra |=1+ |\la Ce,f\ra |^2.
\end{align*}
It follows that $\|X\|\geq 1$, which completes the proof.
\end{proof}

\begin{lemma}\label{L:Rosen}
Let $A\in\SC$. Then
\begin{enumerate}
\item[(i)] $\sigma_\pi(A)=\sigma_\delta(A)=\sigma(A)$;
\item[(ii)] $\sigma(J_A)\subset \frac{1}{2}[\sigma(A)+\sigma(A)]$.
\end{enumerate}
\end{lemma}

\begin{proof}
(i) For any $\lambda\in\bC$, note that $C(A-\lambda)C=(A-\lambda)^*$.
Then $A-\lambda$ is bounded below if and only if so is $(A-\lambda)^*$, which equals that
$A-\lambda$ is surjective. Hence the result follows readily.

(ii)
Note that $A\circ X\in \SSC$ for all $X\in\SSC$. Thus $\SSC$ and $\SC$ are both invariant under $\tau_{A,-A}$.
Thus, by \cite[Corollary 3.20]{Herr89}, $$\sigma(J_A)=\sigma(\frac{1}{2}\tau_{A,-A}|_{\SC})\subset\sigma(\frac{1}{2}\tau_{A,-A} )=\frac{1}{2}[\sigma(A)+\sigma(A)].$$
\end{proof}

By \cite[Theorem 3.19 \& Corollary 3.20]{Herr89}, the following corollary is clear.

\begin{corollary}\label{C:Rosen}
If $A\in\SC$, then $$\sigma_\pi(\frac{1}{2}\tau_{A,-A})=\sigma_\delta(\frac{1}{2}\tau_{A,-A})=\sigma(\frac{1}{2}\tau_{A,-A})=\frac{1}{2}[\sigma(A)+\sigma(A)].$$
\end{corollary}

Note that each $\SCP$ is invariant under $J_T$ for $p\in\{0\}\cup [1,\infty)$. Recall that $\SCP=\SC\cap\BPH$. Denote $J_{T,p}=J_T|_{\SCP}$. We view $J_{T,p}$ as a linear operator on $(\SCP, \|\cdot\|_p)$. By \cite[Theorem 2.3.10]{Ringrose}, $J_{T,p}$ is bounded.

\begin{lemma}\label{L:defect}
If $T\in\SC$, then $ \sigma_\pi(J_{T,p})\subset\frac{1}{2}[\sigma(T)+\sigma(T)]$ for all $p\in\{0\}\cup [1,\infty)$.
\end{lemma}

\begin{proof}
Assume that $z\in \sigma_\pi(J_{T,p})$. Then there exist $\{X_n\}\in \SCP$ with $\|X_n\|_p=1$ for all $n$ such that
$ \|J_{T,p}X_n-zX_n\|_p\rightarrow 0.$
That is, $\|\frac{1}{2}\tau_{T,-T} (X_n)-zX_n\|_p\rightarrow 0$. Thus $\frac{1}{2}\tau_{T,-T}|_{\BPH}-z$ is not bounded below.
By \cite[Theorem 3.54]{Herr89},  we deduce that $\frac{1}{2}\tau_{T,-T} -z$ is not bounded below. In view of Corollary \ref{C:Rosen}, we
have $z\in\frac{1}{2}[\sigma(T)+\sigma(T)]$.
\end{proof}

Given a Banach space $\mathcal{X}$, we let $\mathcal{X}'$ denote its dual. If $T:\mathcal{X}\rightarrow \mathcal{X}$ is a bounded linear operator, we denote by $T'$ the adjoint of $T$ acting on $\mathcal{X}'$.

\begin{lemma}\label{L:dual}
If $T\in\SC$, then $\sigma_\pi(J_T)=\sigma_\delta(J_{T,1})$ and $\sigma_\delta(J_T)=\sigma_\pi(J_{T,1})$.
\end{lemma}

\begin{proof}
It suffices to prove that $J_T$ is similar to the adjoint $J_{T,1}'$ of $J_{T,1}$.

Denote by $\phi$ the isometrical isomorphism of $\SC$ onto $(\SCY)'$ defined by $\phi(K)=\phi_K$,
where \[\phi_K(X)=\tr(XK), \ \ \forall X\in \SCY. \]
Then it suffices to check that $\phi J_T =J_{T,1}' \phi$.

Fix $Z\in \SC$. Denote $J_1=[\phi J_T] (Z)$ and $J_2=[J_{T,1}' \phi] (Z)$. Then $J_i\in \SCY'$, $i=1,2$.
It suffices to prove that $J_1=J_2$.
Since $$[\phi J_T] (Z) =\phi[J_T(Z)]=\phi_{T\circ Z},\ \ \ [J_{T,1}' \phi] (Z) =J_{T,1}'[\phi(Z)]=J_{T,1}'(\phi_Z),$$ for any $X\in\SCY$, we have
\begin{align*}
J_1(X)&=\phi_{T\circ Z}(X)=\tr[(T\circ Z)X]\\
&=\frac{1}{2}\tr(TZX+ZTX) =\frac{1}{2}\tr(TXZ+XTZ)\\
&=\tr[(T\circ X) Z] =\phi_Z(T\circ X)\\
&=\phi_Z[J_{T,1}(X)]=J_2(X).
\end{align*} This shows that $J_1=J_2$.
\end{proof}

Using dual relations among $\SCP (p\in\{0\}\cup[1,\infty))$ (see Proposition \ref{P:SchattenP}), one can prove as in Lemma \ref{L:dual} the following corollary.

\begin{corollary}\label{C:dual}
Let $T\in\SC$ and $1<p,q<\infty$ with $\frac{1}{p}+\frac{1}{q}=1$.
Then
\begin{enumerate}
\item[(i)] $\sigma_\pi(J_{T,0})=\sigma_\delta(J_{T,1})$ and $\sigma_\delta(J_{T,0})=\sigma_\pi(J_{T,1})$.
\item[(ii)] $\sigma_\pi(J_{T,p})=\sigma_\delta(J_{T,q})$ and $\sigma_\delta(J_{T,p})=\sigma_\pi(J_{T,q})$.
\end{enumerate}
\end{corollary}

Now we are going to prove Theorem \ref{T:JMultiply}.

\begin{proof}[Proof of Theorem \ref{T:JMultiply}]
(i) Clearly, $\|T\circ X\|\leq \|T\|\cdot\|X\|$ for all $X\in\SC $. Thus $ \|J_T\|\leq\|T\|$. Noting that $I\in\SC$,  we have $\|T\|=\|J_T(I)\|\leq\|J_T\|$.
Hence $\|J_T\| =\|T\|$.

(ii) Fix $p\in\{0\}\cup[1,\infty)$. We first prove a key claim.

{\it Claim.} $\frac{1}{2}[\sigma(T)+\sigma(T)]\subset [\sigma_\pi(J_{T}) \cap \sigma_\pi(J_{T,p})]$.

Let $\lambda,\mu\in\sigma(T) $. Then $\lambda\in\sigma_\pi(T)$ and, by Lemma \ref{L:Rosen} (i), $ \bar{\mu} \in\sigma_\pi(T^*)$.
We can choose unit vectors $\{e_n,f_n: n\geq 1\}$ such that $(T-\lambda)e_n\rightarrow 0$ and  $(T-\mu)^*f_n\rightarrow 0$ as $n\rightarrow\infty$.

For $n\geq 1$, set $X_n=e_n\otimes f_n+ (Cf_n)\otimes (Ce_n)$.  By Lemma \ref {L:CSeg}, $X_n\in\SC$ and $1\leq \|X_n\|\leq \|X_n\|_p\leq 2$.
By Lemma \ref{L:computation},
\begin{align*}
2J_T(X_n)=TX_n+X_nT&=(Te_n)\otimes f_n+ (TCf_n)\otimes (Ce_n)\\
&~~~~+ e_n\otimes (T^*f_n)+ (Cf_n)\otimes (T^*Ce_n)\\
&=(Te_n)\otimes f_n+ (C T^*f_n)\otimes (Ce_n)\\
&~~~~+ e_n\otimes (T^*f_n)+ (Cf_n)\otimes (C Te_n).
\end{align*}
Note that
$$\lambda X_n=(\lambda e_n)\otimes f_n+ (Cf_n)\otimes [C(\lambda e_n])$$ and
$$\mu X_n=e_n \otimes (\bar{\mu} f_n)+[C(\bar{\mu}f_n)]\otimes (Ce_n).$$
Then, as $n\rightarrow \infty$,
\begin{align*}
2J_T(X_n)-(\lambda+\mu)X_n&=  [(Te_n)\otimes f_n+ (Cf_n)\otimes (C Te_n)-\lambda X_n]\\
 &~~~+ [(C T^*f_n)\otimes (Ce_n)+ e_n\otimes (T^*f_n)-\mu X_n] \\
&= [(T-\lambda) e_n]\otimes f_n+ (Cf_n)\otimes [C (T - \lambda) e_n ] \\
&~~~+ [C (T^*-\bar{\mu})f_n]\otimes (Ce_n)+ e_n\otimes [(T^*-\bar{\mu})f_n ]\overset{\|\cdot\|_p}{\longrightarrow} 0.
\end{align*}  
This shows that $\frac{1}{2}(\lambda+\mu)\in\sigma_\pi(J_{T,p}) \cap \sigma_\pi(J_{T}).$

Since $\lambda,\mu\in\sigma(T) $ can be choose arbitrarily, we deduce that $\frac{1}{2}(\sigma(T)+\sigma(T))\subset\sigma_\pi(J_{T,p}) \cap \sigma_\pi(J_{T})$.
This proves Claim.
%

In view of Lemma \ref{L:Rosen} (ii) and Lemma \ref{L:defect}, we conclude that $$\sigma(J_T)=\sigma_\pi(J_T)=\sigma_\pi(J_{T,p})=\frac{1}{2}[\sigma(T)+\sigma(T)].$$
It follows immediately from Lemma \ref{L:dual} and Corollary \ref{C:dual} that
$$\sigma_\delta(J_T)=\sigma_\delta(J_{T,p})=\frac{1}{2}[\sigma(T)+\sigma(T)].$$

For any bounded linear operator $A$ on a Banach space, note that
$$\sigma_\delta(A)\subset \sigma_r(A)\subset\sigma(A),\quad \sigma_\pi(A)\subset \sigma_l(A)\subset\sigma(A).$$
Hence we conclude the proof.
\end{proof}

\begin{corollary}\label{C:ideal}
Let $T\in\SC$ and $p\in\{0\}\cup [1,\infty)$. Then
 $$\sigma(J_{T,p})=\sigma_\delta(J_{T,p})=\sigma_\pi(J_{T,p})=\frac{1}{2}[\sigma(T)+\sigma(T)].$$
\end{corollary}
%

\begin{remark}
By Theorem \ref{T:JMultiply}, Corollary \ref{C:Rosen} and Corollary \ref{C:ideal}, if $T\in\SC$, then
the spectra, the approximate point spectra and the approximate defect spectra of $J_T$ and $J_{T,p}$ $(p\in\{0\}\cup[1,\infty))$ coincide with that of $\frac{1}{2}\tau_{T,-T}$, all equaling $\frac{1}{2}[\sigma(T)+\sigma(T)]$.
\end{remark}

\begin{corollary}\label{C:equation}
Let $T\in\SC$ and $p\in\{0\}\cup [1,\infty)$. Then the following are equivalent:
\begin{enumerate}
\item[(i)] for any $Y\in \BH$, the operator equation $TX+XT=Y$ has at least one solution in $\BH$;
\item[(ii)] for any $Y\in \BH$, the operator equation $TX+XT=Y$ has exactly one solution in $\BH$;
\item[(iii)] for any $Y\in \SC$, the operator equation $TX+XT=Y$ has at least one solution in $\SC$;
\item[(iv)]  for any $Y\in \SC$, the operator equation $TX+XT=Y$ has exactly one solution in $\SC$;
\item[(iii)] for any $Y\in \SCP$, the operator equation $TX+XT=Y$ has at least one solution in $\SCP$;
\item[(iv)]  for any $Y\in \SCP$, the operator equation $TX+XT=Y$ has exactly one solution in $\SCP$;
\item[(v)] $0\notin\sigma(T)+\sigma(T)$.
\end{enumerate}
\end{corollary}

%

%

%

%
%
%
%

%

\begin{proposition}\label{L:RosenSpec}
If $T\in\SC$, then $\frac{1}{2}[\sigma_p(T)+\sigma_p(T)]\subset\sigma_p(J_T).$
\end{proposition}

\begin{proof}
Now choose $\lambda,\mu\in\sigma_p(T)$. It suffices to prove that $\frac{1}{2}(\lambda+\mu)\in\sigma_p(J_T)$.

Since $\lambda\in\sigma_p(T)$, we can find a unit vector $e\in\cH$ such that $Te=\lambda e$.
On the other hand, note that $C(T-\mu)^*C=T-\mu$. Since $\mu\in\sigma_p(T)$, it follows that  ${\bar \mu}\in\sigma_p(T^*)$ and we can find a unit vector $f\in\cH$ such that $T^*f=\bar{\mu }f$.

Set $X=e\otimes f+ (Cf)\otimes (Ce)$.
Then $X\in\SC$ and, by Lemma \ref{L:CSeg}, $X\ne 0$.
Moreover, we have \begin{align*}
 2 J_T(X)&=TX+XT\\
  &=(Te)\otimes f+(T Cf)\otimes (Ce)+ e\otimes (T^*f)+  (Cf)\otimes (T^*Ce)\\
  &=(\lambda e)\otimes f+(CT^*f)\otimes (Ce)+ e\otimes (\bar{\mu} f)+  (Cf)\otimes (CTe)\\
  &=\lambda( e\otimes f)+(C\bar{\mu}  f)\otimes (Ce)+ \mu( e\otimes  f) +  (Cf)\otimes (C\lambda e)\\
  &=\lambda( e\otimes f)+\mu [(C  f)\otimes (Ce)]+ \mu( e\otimes  f) +  \lambda[(Cf)\otimes (C e)]\\
  &=(\lambda+\mu) X.
\end{align*} Hence $\frac{1}{2}(\lambda+\mu)\in\sigma_p(J_T)$.
\end{proof}

By Theorem \ref{T:JMultiply} and Proposition \ref{L:RosenSpec}, the invertibility of an operator $T\in\SC$ in general does not imply the invertibility or even the injectivity of $J_T$.

\begin{example}
Let $C$ be a conjugation on $\cH$ and
\[D=\begin{bmatrix}
C&0\\
0 & C
\end{bmatrix}\begin{matrix}
  \cH\\ \cH
\end{matrix}.\]  Then one can check that
\[\cS_D=\left\{ \begin{bmatrix}
A&E^*\\
CEC&B
\end{bmatrix}\begin{matrix}
  \cH\\ \cH
\end{matrix}: A,B\in\SC, E\in\BH \right\}.\] Define $$
T=\begin{bmatrix}
I&0\\
0&-I
\end{bmatrix}\begin{matrix}
  \cH\\ \cH
\end{matrix}.$$
Clearly, $T\in\cS_D$, $\sigma(T)=\{1,-1\}$ and, by Theorem \ref{T:JMultiply}, $\sigma(J_T)=\{0,1,-1\}$.
Moreover, one can check that $\sigma_p(J_T)=\{0,1,-1\}$,
\[\ker J_T=\left\{ \begin{bmatrix}
0&E^*\\
CEC&0
\end{bmatrix}\begin{matrix}
  \cH\\ \cH
\end{matrix}:  E\in\BH \right\}\]
\[\ker (J_T-1)=\left\{ \begin{bmatrix}
A&0\\
0&0
\end{bmatrix}\begin{matrix}
  \cH\\ \cH
\end{matrix}:  A\in\SC \right\} \] and
\[\ker (J_T+1)=\left\{ \begin{bmatrix}
0&0\\
0&B
\end{bmatrix}\begin{matrix}
  \cH\\ \cH
\end{matrix}: B\in\SC \right\}.\]
\end{example}

%

\section{Invertible operators in $\SC$}\label{S:invertible}

This section focuses on some topics concerning
Jordan invertible elements of $\SC$.

We remark that an element $T\in\SC$ is Jordan invertible if and only if $T$ is an
invertible operator.
In fact, if $T\in\BH$ is an invertible operator, then it is easy to see that $T^{-1}\in\SC$. One can check that
$Q_T$ is invertible with $Q_T^{-1}=Q_{T^{-1}}$. On the other hand, if $Q_T$ is invertible, then there exists $X\in\SC$ such that
$Q_T(X)=TXT=I$. So $T$ is invertible.

\subsection{Connectedness of the invertibles}
It is well known that the collection of invertible operators in $\BH$ is path connected.
The following result is its analogue in $\SC$.


\begin{theorem}\label{T:invtbCntd}
The set of invertible operators in $\SC$ is path connected.
\end{theorem}

\begin{proof}
Choose an invertible operator $T\in\SC$ and assume that $T=U|T|$ is its polar decomposition. So $U$ is unitary and $|T|$ is invertible.
By \cite[Theorem 2]{Gar07}, $U\in \SC$ and there exists a conjugation $J$ on $\cH$ such that $U=CJ$ and $J|T|=|T|J$.

{\it Claim.} There is an arc $\{ T_\lambda: \lambda\in[0,1]\}$ in $\SC$ connecting $T$ to $U$.

Denote $m=\min\sigma(|T|)$ and $M=\max\sigma(|T|)$.
For $\lambda\in[0,1]$, define $f_\lambda\in C[m,M]$ as $f_\lambda(t)=\lambda+(1-\lambda)t$. Set $T_\lambda=Uf_\lambda(|T|)$, $\lambda\in[0,1]$.
Then the path $\{f_\lambda(|T|):\lambda\in[0,1]\}$ connects $T$ to $U$.
It suffices to prove that $T_\lambda\in\SC$ for each $\lambda\in[0,1]$.

Since $J|T|=|T|J$, it follows from Proposition \ref{P:funcCalcu} that $f_\lambda(|T|)J=Jf_\lambda(|T|)$. Noting that $U^*=U^{-1}=JC$, we obtain
\[CT_\lambda C=CUf_\lambda(|T|)C=Jf_\lambda(|T|)C=f_\lambda(|T|)JC=f_\lambda(|T|)U^*=T_\lambda^*.
\]That is, $T_\lambda\in\SC$.
This proves the claim.

Now it remains to prove that there is an arc $\{ U_\lambda: \lambda\in[0,1]\}$ in $\SC$ connecting the identity operator to $U$.
Since $U$ is unitary, by \cite[Proposition 5.29]{RGDouglas}, the set of unitary operators in $W^*(U)$ is path connected. So there is an arc $\{ U_\lambda: \lambda\in[0,1]\}$ in $W^*(U)$ connecting the identity operator to $U$. On the other hand, since $U\in\SC$, by Proposition \ref{P:funcCalcu}, $W^*(U)\subset \SC$. This shows that $U_\lambda\in\SC$. Thus we complete the proof.
\end{proof}

\begin{proposition}\label{T:UnitCntd}
The set of unitary operators in $\SC$ is path connected.
\end{proposition}

\begin{proof}
Choose a unitary operator $U\in\SC$. Set $J=CU$. Then one can check that $J$ is a conjugation on $\cH$ and $U=CJ$.
By Corollary \ref{C:unitary-cntd}, there exists a path $\{J(t): t\in[0,1]\}$ of conjugations connecting $J$ to $C$.
Then $\{CJ(t): t\in[0,1]\}$ is a path of unitary operators connecting $CJ(=U)$ to $C^2(=I)$. For each $t\in[0,1]$,
note that $CJ(t)$ is a unitary operator lying in $\SC$, since $ C(CJ(t))C=J(t)C=[CJ(t)]^*$. We conclude that $U$ is connected to the identity operator via a path of unitary operators in $\SC$.
\end{proof}

\begin{proposition}\label{T:FredholmCntd}
The set of Fredholm operators in $\SC$ is path connected.
\end{proposition}

\begin{proof}
Assume that $T\in\SC$ is Fredholm. It suffices to find a path of Fredholm operators in $\SC$ connecting $T$ to $I$.
In view of Theorem \ref{T:invtbCntd}, we may directly assume that $T$ is not invertible.
Thus $0$ is an isolated point of $\sigma(|T|)$ and $0<\dim\ker T<\infty$.

By \cite[Theorem 2]{Gar07}, we assume that $T=CJ|T|$, where $J$ is a partial conjugation $J$ acting on $\cH$ and supported on $\overline{\ran |T|}$ such that $J|T|=|T|J$.
Then we may assume $$|T|=\begin{bmatrix}
0&0\\
0& A
\end{bmatrix}\begin{matrix}
\ker |T|\\ (\ker |T|)^\bot
\end{matrix},\ \ \  J=\begin{bmatrix}
0&0\\
0& J_0
\end{bmatrix}\begin{matrix}
\ker |T|\\ (\ker |T|)^\bot
\end{matrix},$$
where $J_0$ is a conjugation on $(\ker |T|)^\bot$.
Clearly, $A$ is invertible.

Choose a conjugation $J_1$ on $\ker |T|$.
Set
$$ \widetilde{J}=\begin{bmatrix}
J_1&0\\
0& J_0
\end{bmatrix}\begin{matrix}
\ker |T|\\ (\ker |T|)^\bot
\end{matrix}, \ \ \ P_\lambda=\begin{bmatrix}
\lambda I_1&0\\
0& A
\end{bmatrix}\begin{matrix}
\ker |T|\\ (\ker |T|)^\bot
\end{matrix}, \ \ \lambda\in [0,1],$$ where $I_1$ is the identity operator on $\ker |T|$.

For $\lambda\in [0,1]$, define $T_\lambda=C\widetilde{J}P_\lambda$. Clearly, $T_0=T$, $T_\lambda$ is invertible for $\lambda\in (0,1]$ and
$$ CT_\lambda C= \widetilde{J}P(\lambda)C=P(\lambda)\widetilde{J}C=T_\lambda^*. $$
Thus $\{T_\lambda: \lambda\in[0,1]\}$ is a path of invertible operators in $\SC$ connecting $T$ to $T_1$ (which is invertible). In view of Theorem \ref{T:invtbCntd}, one can see the conclusion.
\end{proof}

In the  proof  of Proposition \ref{T:FredholmCntd}, one can see that $T_\lambda-T\in\KH$, which implies that
$T_\lambda^{-1}T-I\in\KH$ and $TT_\lambda^{-1}-I\in\KH$ for $\lambda\in(0,1]$. Thus the following corollary is clear.

\begin{corollary}
Let $T\in\SC$. Then the following are equivalent:
\begin{enumerate}
\item[(i)] $T$ is a Fredholm operator;
\item[(ii)] given $\eps>0$, there exists $K\in\KH$ with $\|K\|<\eps$ such that $T+K$ is invertible and $T+K\in\SC$;
\item[(iii)] there exists invertible $A\in\SC$ such that $TA-I,AT-I\in\KH$.
\end{enumerate}
\end{corollary}

\begin{remark}
Note that $I\in\SC$ for any conjugation $C$ on $\cH$.
Then, in view of Theorem \ref{T:invtbCntd}, Proposition \ref{T:UnitCntd} and Proposition \ref{T:FredholmCntd}, we conclude that
the set of invertible operators in $\SH$, the set of unitary operators in $\SH$ and the set of Fredholm operators in $\SH$ are all path connected.
\end{remark}

\subsection{Invertible approximation}
This subsection focuses on invertible approximation in $\SC$,
that is, describing which operators can be approximated in norm by invertible operators in $\SC$.

Denote by $\GH$ the collection of invertible operators in $\BH$.
By \cite[Proposition 10.1]{AFHV}, an operator $T$ lies in the norm closure of $\GH$ if and only if  either $T$ is not a semi-Fredholm operator or $T$ is a semi-Fredholm operator with $\ind~T=0$. Recall that an operator $R$ is called a semi-Fredholm operator if $\ran R$ is closed and either $\dim\ker R$ or  $\dim\ker R^*$ is finite; in this case, $\ind R=\dim\ker R-\dim\ker R^*$ is called the index of $R$.
If, in addition, $-\infty<\ind R<\infty$, then $T$ is called a Fredholm operator.

Let $T\in\SC$. Then $CTC=T^*$ and $\dim\ker T=\dim\ker T^*$.
This shows that if $T$ is a semi-Fredholm operator, then $\ind T=0$.
In view of the invertible approximation in $\BH$, it is natural to conjecture that every operator in $\SC$ is a norm limit of invertible operators in $\SC$.
This is indeed the case.

%

\begin{proposition}\label{P:invertibleAppr}
If $C$ is a conjugation on $\cH$, then $\SC\cap \GH$ is norm dense in $\SC$.
\end{proposition}

\begin{proof}
Choose an operator $T\in\SC$.
By \cite[Theorem 2]{Gar07}, there exists a partial conjugation $J$ supported on $\overline{\ran |T|}$ such that $T=CJ|T|$ and $J|T|=|T|J$. Then
\[|T|=\begin{bmatrix}
0&0\\
0&A
\end{bmatrix}\begin{matrix}
\ker |T| \\ (\ker |T|)^\bot
\end{matrix}, \ \ J=\begin{bmatrix}
0&0\\
0&J_1
\end{bmatrix}\begin{matrix}
\ker |T| \\ (\ker |T|)^\bot
\end{matrix},\] where $A$ is positive, $J_1$ is a conjugation on $\overline{\ran |T|}=(\ker |T|)^\bot$ and $J_1A=AJ_1$.

Fix an $\eps>0$. Assume that $E(\cdot)$ is the projection-valued spectral measure associated with $A$. Set $P=E([0,\frac{\eps}{2}])$.
From Proposition \ref{P:funcCalcu} one can see $J_1P=PJ_1$.
Hence
\[|T|=\begin{bmatrix}
0&0&0\\
0&A_1&0\\
0&0&A_2
\end{bmatrix}\begin{matrix}
\ker |T| \\ \ran P \\(\ker |T|)^\bot\ominus \ran P
\end{matrix}, \ \ J =\begin{bmatrix}
0&0&0\\
0&J_{1,1}&0\\
0&0&J_{1,2}
\end{bmatrix}\begin{matrix}
\ker |T| \\ \ran P \\(\ker |T|)^\bot\ominus \ran P
\end{matrix}.\]

Choose a conjugation $J_0$ on $\ker |T|$ and set
\[Q=\begin{bmatrix}
\frac{\eps I_0}{2}&0&0\\
0&\frac{\eps I_1}{2}&0\\
0&0&A_2
\end{bmatrix}\begin{matrix}
\ker |T| \\ \ran P \\(\ker |T|)^\bot\ominus \ran P
\end{matrix}, \ \ \tilde{J} =\begin{bmatrix}
J_0&0&0\\
0&J_{1,1}&0\\
0&0&J_{1,2}
\end{bmatrix}\begin{matrix}
\ker |T| \\ \ran P \\(\ker |T|)^\bot\ominus \ran P
\end{matrix},\]  where $I_0$ is the identity operator on $\ker |T|$ and $I_1$ is the identity operator on $\ran P$.
Then $Q$ is positive, invertible,  and $\tilde{J}$ is a conjugation on $\cH$ commuting with $Q$. Set $T_\eps=C\tilde{J}Q$. Then $T_\eps$ is invertible, $T_\eps\in\SC$ and
\begin{align*}
\|T_\eps-T\|&\leq \|T_\eps-C\tilde{J}|T|\|+\|C\tilde{J}|T|-CJ|T|\|\\
&= \|T_\eps-C\tilde{J}|T|\|\\
&=\|Q-|T|\|\leq \frac{\eps}{2}.
\end{align*}
Since $\eps$ can be chosen arbitrarily small, we conclude that $T$ is the norm limit of invertible ones in $\SC$.
\end{proof}

By a classical approximation result of Apostol and Morrel (see \cite{ApMo} or \cite[Theorem 6.15]{Herr89}), if an operator $T$ is biquasitriangular (that is, $\ind(T-z)=0$ whenever defined), then $T$ can be approximated in norm by operators with finite spectra.
Note that each operator in $\SC$ is biquasitriangular. Thus it is natural to ask the following question.

\begin{question}
Are those elements with finite spectra norm dense in $\SC$?
\end{question}

\noindent{\bf Acknowledgements}
The second author was supported by National Natural Science Foundation of China (Grant
No. 12171195).\vskip0.1cm

%

\end{document}